\documentclass[10pt]{amsart}
\usepackage[utf8]{inputenc}
\usepackage{amsthm, amssymb}
\usepackage{amsmath}
\usepackage{bm} 
\usepackage{graphicx} 
\usepackage[shortlabels]{enumitem} 
\usepackage[normalem]{ulem} 
\usepackage{xcolor}
\usepackage{hyperref} 
\usepackage{bbm} 

\title[Definable separability and second-countability]{Definable separability and second-countability in o-minimal structures}
\author[Pablo And\'ujar Guerrero]{Pablo And\'ujar Guerrero} 
\address{University of Leeds}
\email{pa377@cantab.net}
\date{}

\newtheorem{theorem}{Theorem}[section]
\newtheorem{definition}[theorem]{Definition}
\newtheorem{lemma}[theorem]{Lemma}
\newtheorem{fact}[theorem]{Fact}
\newtheorem{proposition}[theorem]{Proposition}
\newtheorem{remark}[theorem]{Remark}
\newtheorem{claim}[theorem]{Claim}

\newtheorem{example}[theorem]{Example}
\newtheorem{conjecture}[theorem]{Conjecture}

\newcommand{\Aa}{\mathcal{A}}
\newcommand{\Mm}{\mathcal{M}}

\newcommand{\Bb}{\mathcal{B}}
\newcommand{\Cc}{\mathcal{C}}

\newcommand{\Hh}{\mathcal{H}}

\newcommand{\Ainf}{A^{\text{inf}}}
\newcommand{\Afin}{A^{\text{fin}}}

\newcommand{\tausorg}{\tau_{\text{sorg}}}

\newcommand{\nice}{$\text{DSC}_\tau$}

\newenvironment{claimproof}[1][\proofname]
               {
                 \proof[#1]
                 
               }
               {
                 \endproof
               }

\begin{document}

\begin{abstract}
We show that separability and second-countability are first-order properties of topological spaces definable in o-minimal expansions of \mbox{$(\mathbb{R},<)$}. We do so by introducing first-order characterizations --definable separability and definable second-countability-- which make sense in a wider model-theoretic context. We prove that, within o-minimality, these notions have the desired properties, including their equivalence among definable metric spaces, and conjecture a definable version of Urysohn's Metrization Theorem.  
\end{abstract}

\maketitle

\noindent
{\small \emph{Mathematics Subject Classification 2020.} 03C64 (Primary); 54A05, 54D65, 54D70 (Secondary). } 

\section{Introduction}

In this paper we deepen the area of tame topology by introducing first-order notions of separability and second-countability. 

O-minimality was birthed as a tame topological setting in the context of the o-minimal Euclidean topology~\cite{pillay86}. Since then o-minimal topology has been developed extensively beyond this seminal context. The first example might be Pillay's proof that definable groups have a definable manifold group topology~\cite{pillay88}, with the subsequent investigation by van den Dries of definable manifold spaces~\cite{dries98}, which in recent years has served as a framework for the development of the theory of definable analytic spaces~\cite{gaga23}. On the other hand, the study of topological spaces of definable functions, inevitably linked to the development of o-minimal analysis~\cite{pet-star-01, asch-fisher-11}, has long been another central topic in o-minimality.
The influential research into parametrizations in~\cite{pil-wilkie-06}, which led to the celebrated Pila-Wilkie theorem, motivated the abstraction of the notion of definable normed space~\cite{thomas12}. This was later generalized to the notion of definable metric space~\cite{walsberg-thesis}, and subsequently by the author to the study of o-minimal (explicitly) definable topological spaces~\cite{andujar_thesis}, providing thus a usable framework for the development of o-minimal point-set topology and functional analysis. 

In this paper we introduce and investigate definable analogues of separability and second-countability, in the context of definable topological spaces in o-minimal structures. We observe that these provide non-trivial dividing lines in our context. Specifically, every definable set with the o-minimal Euclidean topology is definably separable and definably second-countable, the definable Sorgenfrey line (Example~\ref{ex:sorg}) is definably separable but not definably second-countable, and any infinite definable set with the discrete topology lacks either property. Our main results (Theorems~\ref{thm:separability} and~\ref{thm:2c}) state that, in o-minimal expansions of $(\mathbb{R},<)$, definable separability and definable second-countability are equivalent to their classical counterparts, similarly to properties such as definable connectedness and definable compactness. We conjecture that, under minimal assumptions, definable second-countability characterizes o-minimal definable topological spaces which are, up to definable homeomorphism, Euclidean (Conjecture~\ref{conj:UMT}).

Our definition of definable separability makes sense in any structure (regardless of o-minimality) and our definition of definable second-countability invokes o-minimal dimension. Our analysis depends largely on basic o-minimality, including the fact that o-minimal dimension coincides with (naive) topological dimension, and on the Fiber Lemma described in Fact~\ref{fact:fiber}. 
Recently, there have been various axiomatic explorations of tame topology~\cite{sim-wal-19, dolich-good-22} in settings which generalize o-minimality, as well as non-strongly-minimal dp-minimal expansions of fields. This literature has yielded, in particular, the existence of a well-behaved notion of dimension in these settings. Hence, it seems plausible that the definitions proposed here extend fruitfully to more general frameworks.


The structure of the paper is as follows. 
In Section~\ref{sec:prelim} we present our notation and o-minimal machinery. Section~\ref{sec:prelim-top} includes our definitions of definable separability and second-countability, as well as some preliminary results showing that the properties behave as expected. In Section~\ref{sec:main} we prove our main results, namely the equivalence with the classical properties in o-minimal expansions of $(\mathbb{R},<)$. We also achieve a deeper characterization of the introduced properties, and derive that they are definable in families. In Section~\ref{sec:metric} we investigate definable separability and second-countability in the context of o-minimal definable metric spaces, proving their equivalence in this context. We also contrast our definitions with a different definition of definable separability for definable metric spaces, presented by Walsberg in~\cite{walsberg-thesis}. Finally in Section~\ref{sec:conj} we conjecture an o-minimal definable version of Urysohn's metrization theorem.   

In~\cite{atw} Thomas, Walsberg and the author proved various results on o-minimal definable topological spaces which can be seen to capture ``definable first-countability". In particular it is shown that every definable topological space in an o-minimal expansion of $(\mathbb{R},<)$ is first-countable. Consequently the present paper does not dwell on this property. 

\subsection*{Acknowledgments}
The author thanks Margaret E. M. Thomas for her supervision of research that led to much of the contents of this paper, which was undertaken during the author's doctoral degree, and Erik Walsberg, who informed the author of the proof of Proposition~\ref{remark_def_sep}. The author also thanks Pantelis Eleftheriou for comments on first drafts, and him, Mervyn Tong, and Ben De Smet, for helpful discussions on the topic, which finally led to the correct definition of definable second-countability. The author thanks the anonymous referee, who pointed out various mistakes and greatly improved the presentation of the paper.

\subsection*{Funding}
During the writing of this paper and the undertaking of the research into its contents the author was supported by the UK Engineering and Physical Sciences Research Council (EPSRC) Grant EP/V003291/1, and by the Graduate School and the Department of Mathematics at Purdue University. 

\section{Preliminaries} \label{sec:prelim}

\subsection{Conventions and terminology} \label{sec:notation}

Throughout we fix a first-order structure $\Mm=(M,\ldots)$. ``Definable" always means ``definable in $\Mm$, possibly with parameters". Every formula we consider is in the language of $\Mm$. We use $n$, $m$ and $k$ to denote positive natural numbers. We use $x$, $y$, $a$, $b$, $\ldots$ to denote tuples of variables or parameters in $M$. Given a tuple $x$, we denote its length by $|x|$.

Recall that a family of definable sets $\Aa=\{A_b : b\in B\}$ is (uniformly) definable if there exists a (partitioned) formula $\varphi(x,y)$ such that $B\subseteq M^{|y|}$ is a definable set and $A_b=\{ a \in M^{|x|} : \Mm\models \varphi(a,b)\}$ for every $b\in B$. 

For any given $n\geq 2$ we denote by $\pi:M^n \rightarrow M^{n-1}$ the projection to the first $n-1$ coordinates, where $n$ will often be clear from context. Given a set $X\subseteq M^n$ and a tuple $a \in M^m$, where $m<n$, let $X_a = \{b\in M^{n-m} : ( a,b ) \in X\}$.

We work mostly in the case where $\Mm$ is an o-minimal expansion of a dense linear order without endpoints $(M,<)$. In this case, given a definable set $X$, we denote its o-minimal dimension by $\dim X$. By the \emph{o-minimal Euclidean topology} we mean the order topology on $M$ and induced product topology on $M^n$. For background o-minimality we direct the reader to~\cite{dries98}. In particular, we will use o-minimal cell decomposition and uniform finiteness. 

Consider a finite collection of definable families of sets $\Bb_i = \{U^i_b : b \in B_i\}$, for $i\leq k$. We define their (definable) disjoint union as follows. For each $a=(a_1, \ldots, a_k) \in M^{k}$, let $i(a)$ denote the largest $i\leq k$ with $a_i \geq a_1$. For each $\mathbf{b}=(b_1,\ldots, b_k) \in \prod_{i=1}^k B_i$, let $V_{(a,\mathbf{b})}=U^{i(a)}_{b_{i(a)}}$. The disjoint union of the families $\Bb_i$ is the definable family $\{V_{(a,\mathbf{b})} : (a,\mathbf{b})\in M^k \times \prod_{i=1}^k B_i \}$. Note that this construction does not invoke new parameters.

\subsection{O-minimal preliminaries}

In this subsection we present our main tools from o-minimality. They are mostly connected to the Fiber Lemma stated in Fact~\ref{fact:fiber}.

In general we will not assume that our structure $\Mm$ has elimination of imaginaries. The purpose of the following technical lemma is to circumvent this fact. 

\begin{lemma}\label{lem:inf-family}
Let $\Mm$ be an o-minimal structure. Let $\Aa=\{ A_b : b\in B\}$ be an infinite definable family of sets. There exists another infinite definable family $\Aa'=\{ A'_t: t \in I\}$, where $I\subseteq M$ is an interval, with the following properties. 
\begin{enumerate}
    \item \label{itm:A-I_3} $\Aa' \subseteq \Aa$.
    \item \label{itm:A-I_2} $A'_t \neq A'_s$ for every distinct pair $t, s \in I$. 
\end{enumerate}
\end{lemma}
\begin{proof}
Let $\Aa$ be as in the lemma. We begin by proving the case where $B \subseteq M$. 

Let $E \subseteq B \times B$ be the definable equivalence relation on $B$ given by $( b, c ) \in E$ whenever $A_b = A_c$. For each $b\in B$, let $E(b,M)$ denote the equivalence class of $b$ by $E$.
Let $B^{(0)}=\{ b \in B :  \inf E(b,M) \in E(b,M)\}$ and $B^{(1)}=B\setminus B^{(0)} = \{ b \in B : \inf E(b,M) \notin E(b,M) \}$. For each $i\in \{0, 1\}$, let $\Aa^{(i)}= \{ A_b : b \in B^{(i)}\}$. Clearly $\Aa = \Aa^{(0)}\cup \Aa^{(1)}$.  Since $\Aa$ is infinite, there must be some fixed $i\in \{0,1\}$ such that $\Aa^{(i)}$ is infinite. 

By o-minimality note that, for every $b, c \in B^{(i)}$, if $\inf E(b,M) = \inf E(c,M)$ then $A_b = A_c$. In particular the definable set $C=\{ \inf E(b,M) : b \in B^{(i)}\}\setminus \{-\infty\}$ is infinite. 
For every $c \in C$, let $A'_c$ be the set $A_b$ where $b \in B^{(i)}$ is any element satisfying that $c = \inf E(b,M)$. Observe that the family $\{ A'_b : b \in C\}$ is definable, infinite, and it satisfies that $A'_b \neq A'_c$ for every distinct pair $b, c \in C$. Finally, again by o-minimality there exists an interval $I\subseteq C$ such that $\{A'_c : c \in I\}$ is infinite. This completes the proof of the case $B\subseteq M$. 

Now let $B\subseteq M^n$. We complete the proof of the lemma by showing the existence of an infinite definable subfamily of $\Aa$ indexed definably by a subset of $M$. We proceed by induction on $n$. The base case $n=1$ is trivial. We assume that $n>1$.  

Consider the projection $\pi(B)$ of $B$ to the first $n-1$ coordinates. 
Suppose that there exists some $d \in \pi(B)$ such that the definable family $\Aa_d=\{ A_{( d,t )} : t \in B_d\}$ is infinite. Then it suffices to pass to this family. Hence onwards suppose that, for every $d \in \pi(B)$, the family $\Aa_d$ is finite. For each $d\in \pi(B)$ and $t \in B_d$ let $E_d(t,M)=\{ s \in B_d : A_{( d, t)} = A_{( d, s)}\}$, and set $C_d = \{ \inf E_d(t,M) : t \in B_d\} \setminus \{-\infty\}$.
Observe that the sets $C_d$ are definable uniformly in $d$. Moreover, by o-minimality, for every $d\in \pi(B)$ we get that $|C_d|\leq |\Aa_d|<\infty$ and $|\Aa_d| \leq 1 + 2 |C_d|$. By uniform finiteness we derive that there exists a fixed $m$ such that $|\Aa_d| \leq m$ for every $d \in \pi(B)$. 

By o-minimality for every $d\in \pi(B)$ the sets in $\Aa_d$ are linearly ordered by the relation $<_d$, given by $A_{( d,s )} <_d A_{( d,t )}$ whenever there is $t' \in B_d$ with $A_{( d,t' )} = A_{( d,t )}$ and $s' < t'$ for every $s' \in B_d$ with $A_{( d,s' )} = A_{( d,s )}$. For each $1 \leq i\leq m$ let $A'_d(i) \in \Aa_d$ denote either the $i$-th set in $\Aa_d$ with respect to the order $<_d$ or, in the case where $|\Aa_d|<i$, the maximum one. Observe that, for each $1 \leq i\leq m$, the family of sets $\Aa(i)=\{A'_d(i) : d \in \pi(B)\}$ is definable. By choice of $m$ we also have that $\cup_{i\leq m} \Aa(i) = \Aa$. Since $\Aa$ is infinite, it follows that $\Aa(i)$ is infinite for some fixed $i$. We apply the induction hypothesis to $\Aa(i)$ to complete the proof.   
\end{proof}

In Lemma~\ref{lem:inf-family} above observe that, whenever the sets in $\Aa$ are not unary, it might be that $\Aa'$ is not definable over the same parameters as $\Aa$.

We now present the Fiber Lemma for o-minimal dimension. 

\begin{fact}[\cite{dries98}, Chapter 4, Proposition 1.5 and Corollary 1.6]\label{fact:fiber}
Let $\Mm$ be an o-minimal structure. Let $X\subseteq M^{n}\times M^{m}$ be a definable set and, for any $d\in \{-\infty, 0, 1, \ldots, m\}$, let $X(d)=\{ a \in M^n : \dim X_a = d\}$. Then the sets $X(d)$ are definable and  
\[ 
\dim \left( \bigcup_{a\in X(d)} \{a\} \times X_a \right) = \dim(X(d)) + d.
\]
In particular $\dim X = \max_{0\leq d \leq m} \dim(X(d)) + d$. 
\end{fact} 

We derive some easy consequences of the above fact in the next two lemmas. 

\begin{lemma}\label{lem:disjoint-union-dim}
Let $\Mm$ be an o-minimal structure. Let $\{A_b : b\in B\}$ be an infinite definable family of pairwise disjoint sets with $\dim A_b \geq n$ for every $b\in B$. Then 
$
\dim (\cup_{b\in B} A_b) > n.  
$
\end{lemma}
\begin{proof}
By Lemma~\ref{lem:inf-family}, after passing to a subfamily if necessary, we may assume that $B$ is an interval and that $A_b \cap A_c =\emptyset$ for every distinct pair $b, c \in B$. In particular the set $X=\cup_{b\in B} \{b\} \times A_b$ is in definable bijection with $A=\cup_{b\in B} A_b$, by means of the projection $( b, a ) \mapsto a$. So $\dim A = \dim X$. By Fact~\ref{fact:fiber} note that $\dim X \geq 1 + n$, and the result follows.  
\end{proof}

\begin{lemma}\label{fact:fiber-max}
Let $\Mm$ be an o-minimal structure. Let $X\subseteq M^n \times M^m$ be a non-empty definable set. For every $a\in M^n$  and $c\in M^m$ let $X(a,M)=X_a=\{c' \in M^m : ( a,c' ) \in X\}$ and $X(M,c)=\{a' \in M^n : ( a',c ) \in X\}$. 
Let $A$ and $C$ denote the projections of $X$ to the first $n$ and last $m$ coordinates respectively.

Suppose that $\dim X(M,c) = \dim A$ for every $c\in C$. Then 
\[
\dim C = \max_{a \in A} \dim X(a,M). 
\]
\end{lemma}
\begin{proof}
Let $f:M^n \times M^m \rightarrow M^m \times M^n$ be the permutation map given by $f(a,c)=(c,a)$. 
By Fact~\ref{fact:fiber}, applied to the set $X^{\text{opp}}=f(X)$, and the fact that $\dim X(M,c) = \dim A$ for every $c\in C$, we have that $\dim X = \dim A + \dim C$. On the other hand, by Fact~\ref{fact:fiber} applied to the set $X$, we have that 
\[
\dim X \leq \dim A + \max_{a \in A} \dim X(a,M) \leq \dim A + \dim C,
\]
and the lemma follows. 
\end{proof}

\section{Topological definitions and basic results} \label{sec:prelim-top}

In this section we introduce our definitions of definable separability and definable second-countability.
The former was inspired by a similar notion introduced by Walsberg in~\cite{walsberg-thesis} in the o-minimal definable metric setting. (We explain the connection between both definitions in Section~\ref{sec:metric}.) We also prove some preliminary results. Recall that a topological space is separable if it has a countable dense subset, and second-countable if the topology has a countable basis. 

Our setting will be that of definable (explicitly in the sense of Flum and Ziegler~\cite{FZ}) topological spaces. 

\begin{definition}\label{def:dts}
A \emph{definable topological space} $(X,\tau)$, $X\subseteq M^n$, is a topological space such that there exists a definable family of subsets of $X$ that is a basis for $\tau$. 
\end{definition}  

Let $\Mm$ be o-minimal. Any definable subset of $M^n$ is a definable topological space with o-minimal Euclidean topology. Further examples within o-minimality include the definable manifold spaces studied by Pillay~\cite{pillay88} and van den Dries~\cite[Chapter 10]{dries98}, the definable Euclidean quotient spaces of van den Dries~\cite[Chapter 10]{dries98} and Johnson~\cite{johnson14}, the definable normed spaces of Thomas~\cite{thomas12}, and the definable metric spaces of Walsberg~\cite{walsberg15}. Peterzil and Rosel~\cite{pet_rosel_18} studied one-dimensional definable topologies. See the author's doctoral dissertation~\cite{andujar_thesis} for an exhaustive exploration of o-minimal definable topological spaces.

Given a topological space $(X,\tau)$ and a subset $Y\subseteq X$ we denote its closure in the topology $\tau$ by $cl(X)$. 
We say that a definable topological space $(X,\tau)$ has \emph{small boundaries} if every definable set $Y\subseteq X$ satisfies that $\dim (cl(Y) \setminus Y) < \dim Y$. Any o-minimal Euclidean space has small boundaries~\cite[Chapter 4, Theorem 1.8]{dries98}. 

We now present our definition of definable separability. Observe that the notion makes sense in any structure, regardless of o-minimality. 

\begin{definition}[Definable separability]\label{dfn:separable}
A definable topological space $(X,\tau)$ is \emph{definably separable} if there exists no infinite definable family of open pairwise disjoint sets in $\tau$.
\end{definition}

The reader will note the similarity between Definition~\ref{dfn:separable} and the countable chain condition (\textbf{ccc}, or Suslin's condition) for topological spaces, i.e. the condition that a space does not contain an uncountable family of pairwise disjoint open sets. Every separable topological space has the \textbf{ccc}, but the converse is not true. Our main result Theorem~\ref{thm:separability} shows that both being separable and having the \textbf{ccc} are equivalent to being definably separable among topological spaces definable in o-minimal expansions of $(\mathbb{R},<)$, and so ``definable separability" could also be labelled the ``definable countable chain condition".

Clearly any finite definable topological space is definably separable. On the other hand any infinite definable set $X$ with the discrete topology is not definably separable, since $\{\{x\} : x\in X\}$ would be an infinite definable family of pairwise disjoint open sets. The fact that any definable set in an o-minimal structure with the o-minimal Euclidean topology is definably separable (Remark~\ref{remark:separable_euclidean_topology} below) will follow from showing that these spaces are definably second-countable and that this property implies definable separability. We leave it to the interested reader to check that this can also be derived using o-minimal cell decomposition and Lemma~\ref{lem:disjoint-union-dim}.

The following easy lemma will be used in Section~\ref{sec:main}.

\begin{lemma}\label{lem:sep-open-subspace}
Let $(X,\tau)$ be a definably separable definable topological space. Let $Y\subseteq X$ be a definable open subset. Then the subspace $(Y,\tau|_Y)$ is definably separable. 
\end{lemma}
\begin{proof}
Suppose that $(Y,\tau|_Y)$ is not definably separable, witnessed by an infinite definable family of pairwise disjoint open sets $\Aa$. Since $Y$ is open we have that $\tau|_Y \subseteq \tau$, and so $\Aa$ also witnesses that $(X,\tau)$ is not definably separable. 
\end{proof}

In Section~\ref{sec:metric} we discuss the fact that Lemma~\ref{lem:sep-open-subspace} does not hold in general if we drop the assumption that $Y$ is open. Namely, the condition of being definably separable is not hereditary.

We now move on to our notion of definable second-countability. 

\begin{definition}[Definable second-countability] \label{dfn:2nd-count}
Let $\Mm$ be an o-minimal structure and let $(X,\tau)$ be a definable topological space. We say that $(X,\tau)$ is \emph{definably second-countable} if there exists a definable basis $\{U_b: b\in B\}$ for the topology $\tau$ satisfying that, for $x\in X$ and $b\in B$, if $x \in U_b$ then  
\[
\dim \{ c \in B : x \in U_c \subseteq U_b\} = \dim B. 
\]
We write that $\{U_b: b\in B\}$ \emph{witnesses} that $(X,\tau)$ is definably second-countable. 
\end{definition}

It is easy to see that any infinite definable discrete space is not definably second-countable. We present a less obvious non-example. 

\begin{example}\label{ex:sorg}
Let $\Mm$ be o-minimal and consider the \emph{definable Sorgenfrey line} $(M,\tausorg)$, given by the set $M$ with the topology with basis of right half-open intervals $[x,y)$, for $x, y \in M$. When $(M,<)=(\mathbb{R},<)$ this is the classical Sorgenfrey line. 

Since every non-empty definable open set in $(M,\tausorg)$ is infinite it is easy to see, using Lemma~\ref{lem:disjoint-union-dim}, that $(M,\tausorg)$ is definably separable. On the other hand $(M,\tausorg)$ is not definably second-countable. To see this suppose, towards a contradiction, that there exists a definable basis $\Bb=\{U_b : b\in B\}$ witnessing that $(M,\tausorg)$ is definably second-countable. Then, for every $x \in M$, the definable set $B(x)=\{ b \in B : x \in U_b \subseteq [x,\infty)\}$ satisfies that $\dim B(x) = \dim B$. Moreover, the definable family $\{ B(x) : x\in M\}$ satisfies that $B(x) \cap B(y) = \emptyset$ for every distinct $x,y\in M$. A contradiction follows from Lemma~\ref{lem:disjoint-union-dim}. 

Following the same arguments as above, the Sorgenfrey line topology restricted to any interval is still definably separable and not definably second-countable.
\end{example}

The space $(M,\tausorg)$, as well as other spaces that fail to be definably second-countable, arise naturally when considering the pointwise convergence topology on definable families of functions. For example $(M,\tausorg)$ is the topology induced on $M$ by the collection of characteristic functions $\{ \mathbbm{1}_{(x,\infty)} : x \in M\}$ with the pointwise convergence topology, given that we identify each characteristic function $\mathbbm{1}_{(x,\infty)}$ with its index parameter $x$.

In order to work with Definition~\ref{dfn:2nd-count} we require a generalization of the introduced notion that applies to definable subsets of a definable topological space and which is maintained after passing to finite unions. Sadly, this is not the case with the property of being definably second-countable in the subspace topology. This generalization will involve the following definition.

\begin{definition}
Let $(X,\tau)$ be a topological space and $Y\subseteq X$. A \emph{$\tau$-basis for $Y$} is a family of open sets $\Bb$ satisfying that, for every $y\in Y$ and open set $A$ with $y\in A$, there exists some $U \in \Bb$ such that $y \in U \subseteq A$. In other words, $\Bb$ is a family of open sets which contains, for every $y\in Y$, a basis of open neighborhoods of $y$.
\end{definition}

It is easy to see that, if a topological space $(X,\tau)$ admits a countable cover of subsets which admit each a countable $\tau$-basis, then $(X,\tau)$ is second-countable. This is in contrast with the fact that there exist non-second-countable topological spaces which can be partitioned into finitely many second-countable subspaces (see Example A.16 in \cite{andujar_thesis}). 

In the definable context the landscape is analogous. In particular the aforementioned Example A.16 in \cite{andujar_thesis} is definable in $(\mathbb{R},+,\cdot,<)$, and (by Theorem~\ref{thm:2c}) it fails to be definably second-countable, while admitting a partition into two definably second-countable subspaces. We introduce a definable analogue of the property of having a countable $\tau$-basis. We will rely on this property to prove that definable topological spaces are definably second-countable, by induction arguments relying on o-minimal dimension. 

\begin{definition}\label{dfn:nice}
Let $\Mm$ be an o-minimal structure. Let $(X,\tau)$ be a definable topological space and $Y\subseteq X$ be a subset. We say that $Y$ is \emph{\nice{}}\footnote{A more descriptive name for being \nice{} would be being \emph{definably $\tau$-second-countable}, however the chosen more succinct terminology seems to make the text more readable.} if it is definable and there exists a definable $\tau$-basis $\{ U_b : b \in B\}$ for $Y$ satisfying that, for every $y\in Y$ and $b\in B$ with $y\in U_b$, 
\[
\dim \{ c\in B : y\in U_c \subseteq U_b\} = \dim B. 
\]
We write that $\{ U_b : b \in B\}$ \emph{witnesses} that $Y$ is \nice{}. 
\end{definition}

Clearly a definable topological space $(X,\tau)$ in an o-minimal structure is definably second-countable if and only if $X$ is \nice{}. 
The use of Definition~\ref{dfn:nice} in this paper relies on the characterization provided by Lemma~\ref{lem:nice} below. We first present an easy remark. 

\begin{remark}\label{rem:nice}
Let $\Mm$ be an o-minimal structure, $(X,\tau)$ be a definable topological space, $Y\subseteq X$ be definable subset and $\Bb=\{U_b : b \in B\}$ be a definable $\tau$-basis for $Y$. By definition of $\tau$-basis, observe that $\Bb$ witnesses that $Y$ is \nice{} if an only if, for every $y\in Y$ and every definable neighborhood (in $\tau$) $A$ of $y$, it holds that $\dim\{b \in B : y \in U_b \subseteq A\} = \dim B$. 
\end{remark}

\begin{lemma}\label{lem:nice}
Let $\Mm$ be an o-minimal structure. Let $(X,\tau)$ be a definable topological space. Any finite union of \nice{} subsets of $X$ is \nice{}. In particular $(X,\tau)$ is definably second-countable if and only if it admits a finite covering by \nice{} sets.      
\end{lemma}
\begin{proof}
By following an inductive argument in the general case, it suffices to prove that the union of two \nice{} sets is \nice{}. Hence
let $X_1$ and $X_2$ denote two \nice{} subsets of $X$. For each $i\in \{1,2\}$, let $\Bb_i=\{U^i_b : b\in B_i\}$ be a definable $\tau$-basis for $X_i$ witnessing that $X_i$ is \nice{}. 

Let $m=\max \{\dim B_1, \dim B_2\}$. If there exists some fixed $i\in \{1,2\}$ with $\dim B_i <m$, let $B'_i = B_i \times M^{m-\dim B_i}$ and let $V^i_{(b,c)} = U^i_b$ for each $(b,c)\in B'_i$ with $b \in B_i$. Applying Fact~\ref{fact:fiber}, it is easy to see that $\dim B'_i=m$, and furthermore that $\{ V^i_{b'} : b'\in B'_i\}$ still witnesses that $X_i$ is \nice{}. Hence, by passing from $\Bb_i$ to $\{ V^i_{b'} : b'\in B'_i\}$ if necessary, we may assume that $\dim B_1 = \dim B_2$. 

Now observe that the disjoint union (described in Section~\ref{sec:notation}) of $\Bb_1$ and $\Bb_2$ is a $\tau$-basis for $X_1 \cup X_2$, and furthermore, by Remark~\ref{rem:nice} (and using for example Fact~\ref{fact:fiber}), it witnesses that $X_1 \cup X_2$ is \nice{}.
\end{proof}

We now present some simple facts about definable second-countability. 

\begin{proposition} \label{prop:bf} 
Let $\Mm$ be an o-minimal structure. The following hold. 
    \begin{enumerate}[(1)]
        \item \label{itm:prop-bf-1} A definable subspace of a definably second-countable definable topological space is also definably second-countable.  
        \item \label{itm:prop-bf-2} Any finite subset of a definable topological space $(X,\tau)$ is \nice{}. In particular any finite definable topological space is definably second-countable. 
        \item \label{itm:prop-bf-3} Any definable set with the o-minimal Euclidean topology is definably second-countable.
        \item \label{itm:prop-bf-manifold} If there exists a finite covering $\{X_1, \ldots, X_n\}$ of $X$ by definable open subsets, each of which is definably second-countable in the subspace topology, then $(X,\tau)$ is definably second-countable.
    \end{enumerate}
\end{proposition}
\begin{proof}
Let $(X,\tau)$ be a definable topological space. 

We begin by proving \ref{itm:prop-bf-1}. Hence suppose that $(X,\tau)$ is definably second-countable, witnessed by a definable basis $\Bb=\{U_b : b\in B\}$. Let $Y$ be a definable subset of $X$. Observe that the family $\{Y \cap U_b : b\in B\}$ is a basis for the subspace topology $\tau|_Y$ that witnesses that $(Y,\tau|_Y)$ is definably second-countable.   

We now prove \ref{itm:prop-bf-2}. 
By Lemma~\ref{lem:nice}, it suffices to prove that every singleton subset of $(X,\tau)$ is \nice{}. Hence fix $x\in X$, and let $\{ U_b : b \in B_x\}$ be a definable basis of open neighborhoods of $x$ (in particular a definable $\tau$-basis for $\{x\}$).
For each $b\in B_x$ let
$
B_{x,b} = \{ c \in B_x : x \in U_c \subseteq U_b \}. 
$
Pick some $b_x \in B_x$ satisfying that the set $B_{x,b_x}$ is of minimum dimension among all sets $B_{x,b}$ with $b \in B_x$.  
Clearly the definable family $\{ U_b : b\in B_{x,b_x}\}$ is a basis of neighborhoods of $x$. Moreover by choice of $b_x$ we have that, for every $b\in B_{x,b_x}$, it holds that 
\[
\dim \{ c \in B_{x,b_x} : x \in U_c \subseteq U_b\} = \dim B_{x,b} = \dim B_{x,b_x}.
\]
So $\{x\}$ is \nice{} as desired. 

We prove \ref{itm:prop-bf-3}. By \ref{itm:prop-bf-1} it suffices to show that, for any $n$, the space $M^n$ with the o-minimal Euclidean topology is definably second-countable. 

Consider the definable basis for the o-minimal Euclidean topology on $M^n$ given by boxes, namely $\Bb=\{(a_1, b_1) \times \cdots \times (a_n,b_n) : a_i, b_i\in M,\, a_i <b_i, \text{ for } i\leq n\}$. Let $B= \{ (a_1,b_1, \ldots, a_n,b_n) : a_i, b_i\in M,\, a_i <b_i, \text{ for } i\leq n\}$ and, for any $\bm{b}=(a_1,b_1, \ldots, a_n,b_n) \in B$, let $U(\bm{b}) = (a_1, b_1) \times \cdots \times (a_n,b_n)$. Note that $\dim B=2n$.
Now observe that, for every $x=(x_1,\ldots,x_n)\in M^n$ and $\bm{b}=(a_1,b_1, \ldots, a_n,b_n)\in B$, it holds that 
\begin{align}\label{eqn:EU-2C}
\{ \bm{c}\in B : x \in U(\bm{c})\subseteq U(\bm{b})\}=  
&\{  (a'_1,b'_1,\ldots, a'_n,b'_n) : a'_i, b'_i \in M, \notag\\
&\, a_i \leq a'_i < x_i < b'_i \leq b_i, \text{ for } i\leq n\}, \tag{$\dagger$}
\end{align}
and so, whenever $x\in U(\bm{b})$, the set in~\eqref{eqn:EU-2C} has dimension $2n = \dim B$. It follows that $M^n$ with the o-minimal Euclidean topology is definably second-countable.

Finally, we prove~\ref{itm:prop-bf-manifold}. Let $\{ X_i : i \leq n\}$ be a finite covering of $X$ by definable open subsets and, for each $i\leq n$, let $\Bb_i$ be a definable basis for the subspace topology $\tau|_{X_i}$ that witnesses that $(X_i,\tau|_{X_i})$ is definably second-countable. For any $i \leq n$, since $X_i$ is open, the family $\Bb_i$ is a $\tau$-basis for $X_i$, which furthermore witnesses that $X_i$ is \nice{}. By Lemma~\ref{lem:nice} we conclude that $(X,\tau)$ is definably second-countable. 
\end{proof}

Observe that~\ref{itm:prop-bf-3} and~\ref{itm:prop-bf-manifold} in Proposition~\ref{prop:bf} above imply that the definable manifold spaces of van den Dries~\cite[Chapter 10]{dries98} are definably second-countable.

We now show that definable second-countability implies definable separability.

\begin{proposition} \label{prop:2c-implies-sep}
Let $\Mm$ be an o-minimal structure. Let $(X,\tau)$ be a definable topological space. If $(X,\tau)$ is definably second-countable then it is definably separable. 
\end{proposition}
\begin{proof}
We assume that $(X,\tau)$ is definably second-countable and not definably separable and reach a contradiction. Hence let $\{ U_c: c \in C\}$ be a definable basis that witnesses that $(X,\tau)$ is definably second-countable, and let $\Aa=\{A_b: b\in B_{\Aa}\}$ be an infinite definable family of pairwise disjoint open sets. We may clearly assume that $\emptyset \notin \Aa$.

For each $b\in B_{\Aa}$, let $C(b)=\{ c \in C : \emptyset \neq U_c \subseteq A_b\}$. Observe that the definable family $\{C(b) : b\in B_{\Aa}\}$ is infinite and contains pairwise disjoint sets. By Lemma~\ref{lem:disjoint-union-dim} we may fix some $b\in B_{\Aa}$ with $\dim C(b) < \dim C$.  Since $\emptyset \notin \Aa$ then $A_b \neq \emptyset$. Let us fix some $x\in A_b$ and, applying the definition of basis, some $c \in C$ with $x \in U_c \subseteq A_b$. Then it holds that 
\[
0\leq \dim \{c' \in C : y \in U_{c'} \subseteq U_c\} \leq \dim C(b) < \dim C.
\]
However this contradicts that $\{ U_c: c \in C\}$ witnesses that $(X,\tau)$ is definably second-countable.
\end{proof}

\begin{remark}\label{remark:separable_euclidean_topology}
It follows from Propositions~\ref{prop:bf}\ref{itm:prop-bf-3} and~\ref{prop:2c-implies-sep} that any definable set in an o-minimal structure with the o-minimal Euclidean topology is definably separable. By the sentence below Proposition~\ref{prop:bf}, the same is true for definable manifold spaces.
\end{remark}  

\section{Equivalence with the classical properties}\label{sec:main}

In this section we prove our main results. Primarily, we show that definable separability and definable second-countability are equivalent to their classical counterparts among definable topological spaces in o-minimal expansions of $(\mathbb{R},<)$. We divide this result into Theorems~\ref{thm:separability} and~\ref{thm:2c}. We also show in Proposition~\ref{prop:families} that, given a definable family of topological spaces (e.g. a definable family of subspaces of a fixed definable topological space), the subfamilies of those that are definably separable and those that are definably second-countable are both definable (i.e. the properties of definable separability and definable second-countability are definable in families). 

The following lemma is our main tool in proving Theorem~\ref{thm:separability}.

\begin{lemma}\label{lem:sep-dense}
Let $(X,\tau)$ be a definably separable definable topological space. Let $\Aa=\{A_b : b\in B\}$ be a definable family of open subsets of $X$ and let $n$ be such that $\dim A_b \leq n$ for every $b\in B$. Then there exists a definable open set $Z\subseteq \cup \Aa$ such that $\dim Z \leq n$ and $\cup \Aa \subseteq cl(Z)$. (In particular if $(X,\tau)$ has small boundaries then $\dim \cup\Aa \leq n$.)

Moreover, $Z$ can be chosen to be definable over the same parameters as the family $\Aa$ and the topology $\tau$. 
\end{lemma}
\begin{proof}
We fix $(X,\tau)$ and $\Aa=\{A_b : b\in B\}$ as in the lemma. We proceed by induction on $k$, where $B \subseteq M^k$. The bulk of the proof deals with the case $k=1$. The fact that $Z$ can be chosen to be definable over the same parameters as $\Aa$ and $\tau$ will follow immediately from the proof. 

\textbf{Case $k=1$.} For every $x \in \cup\Aa$ let $B(x) = \{ b \in B : x \in A_b \}$. We partition $\cup\Aa$ into two sets as follows. Let $\Ainf=\{ x \in \cup\Aa : B(x) \text{ is infinite}\}$ and $\Afin= \cup\Aa \setminus \Ainf =\{ x \in \cup\Aa : B(x) \text{ is finite}\}$. By o-minimality these sets are definable. 

\begin{claim}\label{claim:Ainf-open}
The set $\Ainf$ is open and $\dim \Ainf \leq n$.
\end{claim}
\begin{claimproof}
The fact that $\dim \Ainf \leq n$ follows from Lemma~\ref{fact:fiber-max} applied to the set $\{ ( b,x ) \in B \times \Ainf : x\in A_b \}$. We show that $\Ainf$ is open. 

By o-minimal uniform finiteness there exists \mbox{$m = \max \{ |B(x)| : x \in \Afin\} < \infty$}. Fix $x\in \Ainf$. Since $B(x)$ is infinite we may pick $m+1$ parameters $b_0, \ldots, b_m \in B(x)$. Consider the open set $U= \cap_{0 \leq i \leq m} A_{b_i}$. Then we have that $x \in U$ and moreover every $y \in U$ satisfies that $|B(y)|>m$, which, by definition of $m$, implies that $B(y)$ is infinite, and so $y \in \Ainf$. So $x \in U \subseteq \Ainf$. Hence $\Ainf$ is open. 
\end{claimproof}

Now, for every $b \in B$, let $A'_b = \{ x \in A_b \cap \Afin : b = \min B(x)\}$. We define 
\begin{equation} \tag{$\dagger$} \label{eqn:sep-dense}
Z = \Ainf \cup  \bigcup\{ int(A'_b) : b \in B \}.
\end{equation}

\begin{claim}
The set $Z$ is open, definable, and satisfies that $\dim Z \leq n$.
\end{claim}
\begin{claimproof}
It is clear that $Z$ is definable and, by Claim~\ref{claim:Ainf-open}, open. 

Observe that the family of sets $\{ A'_b : b \in B\}$ is definable and pairwise disjoint. Hence the same is clearly true of the family $\{ int(A'_b) : b \in B\}$ and so, by definable separability of $(X,\tau)$, this latter family is finite. In particular, since $\dim A_b \leq n$ for every $b\in B$, we have that $\dim \cup\{ int(A'_b) : b \in B\} \leq n$. Applying Claim~\ref{claim:Ainf-open} we conclude that $\dim Z \leq n$. 
\end{claimproof}

We conclude the proof of the case $k=1$ by showing that $\cup\Aa \subseteq cl(Z)$.
Towards a contradiction suppose that $\cup\Aa \setminus cl(Z) \neq \emptyset$. Since $\Ainf \subseteq Z$, note that $\cup\Aa \setminus cl(Z) \subseteq \Afin$. Hence, by uniform finiteness, we may pick and fix $x \in \cup\Aa \setminus cl(Z)$ satisfying that $|B(x)|= \max \{ |B(y)| : y \in \cup\Aa \setminus cl(Z)\}$. Now, since $x \notin cl(Z)$, there exists an open neighborhood $U$ of $x$ such that $U \cap Z = \emptyset$. Since $x \in \Afin$, we may pick $U$ so that it also satisfies that $x \in U \subseteq \cap_{b \in B(x)} A_b$. Note that every $y \in U$ satisfies that $y \notin cl(Z)$ and $B(x) \subseteq B(y)$, and so by maximality of $|B(x)|$ we have that $B(y) = B(x)$. Finally, let $b= \min B(x)$. Since $x\notin int(A'_b) \subseteq Z$ there must exist $y \in U$ such that $y \in A_b \setminus A'_b$. However, by definition of the set $A'_b$, this means that $\min B(y) < b$, which contradicts that $B(y) = B(x)$.     

\textbf{Case $k>1$.} 
Consider the projection $\pi(B)$ of $B$ to the first $k-1$ coordinates.
For every $c\in \pi(B)$, consider the definable family $\Aa_c=\{ A_{c,t} : t\in B_c\}$. By the case $k=1$ there exists a definable open set $Z_c\subseteq \cup\Aa_c$ such that $\dim Z_c \leq n$ and $\cup\Aa_c \subseteq cl(Z_c)$.  
By the definition of the set $Z$ in the proof of the case $k=1$ note that the sets $Z_c$ may be chosen to be definable uniformly in $c\in\pi(B)$ (over the same parameters as the family $\Aa$ and the topology $\tau$). 

We now apply the induction hypothesis to the definable family $\{ Z_c : c\in \pi(B)\}$, and derive that there exists a definable open set $Z \subseteq \cup_{c \in \pi(B)} Z_c$ with $\dim Z\leq n$ and $\cup_{c\in \pi(B)} Z_c \subseteq cl(Z)$. In particular 
\begin{equation*}
\cup \Aa \subseteq \bigcup_{c\in \pi(B)} cl (Z_c) \subseteq cl\left(\bigcup_{c\in \pi(B)} Z_c \right) \subseteq cl (Z).
\end{equation*}
\end{proof}

We illustrate Lemma~\ref{lem:sep-dense} through the following example. 

\begin{example}\label{ex:sep-dense}
For any $n>1$ consider the following non-Hausdorff definable topology on $M^n$. Fix a parameter $a \in M^{n-1}$. For every $x \in M^n$, basic open neighborhoods of $x$ are sets of the form   
$\{x\} \cup ((b,\infty) \times \{ a \}) \subseteq M^n$, for $b \in M$.
Since any two non-empty open sets have non-empty intersection, it follows easily that this space is definably separable. 

For every $x\in M^n$, consider the open set $A_x = \{x\} \cup (M \times \{a\}) \subseteq M^n$. The family $\{A_x : x \in M^n\}$ is definable and $\dim A_x =1$ for every $x \in M^n$. On the other hand $\cup_{x\in M^n} A_x = M^n$, so $\dim (\cup_{x\in M^n} A_x) = n$. Nevertheless, the definable open subspace $M\times \{a\}$ is dense and one-dimensional. 
\end{example}

Example~\ref{ex:sep-dense} shows that we cannot strengthen Lemma~\ref{lem:sep-dense} by changing its conclusion to $\dim \cup\Aa \leq n$ (unless the space has small boundaries). On the other hand, this example is non-Hausdorff. Since there exist Hausdorff definable topological spaces without small boundaries (e.g. the split interval), it is open whether this strengthening of the lemma holds for all Hausdorff spaces.

Through the next lemma we point out that the construction in the proof of Lemma~\ref{lem:sep-dense} yields that the property of being definably separable can be expressed with a single first-order sentence. We will use this fact in Proposition~\ref{prop:families} to show that definable separability is definable in families.  

\begin{lemma}\label{lem:sep-X(n)}
Let $\Mm$ be an o-minimal structure. Let $(X,\tau)$ be a definable topological space, for which we fix a definable basis $\Bb=\{ U_b : b \in B\}$. For every $n$, let $\Bb(n) \subseteq \Bb$ be the definable family of basic open sets of dimension at most $n$, and set $X(n)=\cup\Bb(n)$.

The space $(X,\tau)$ is definably separable if and only if, for every $n\leq \dim X$, there exists a definable open set $Z(n) \subseteq X(n)$ such that $X(n) \subseteq cl(Z(n))$ and $\dim Z(n) \leq n$.
\end{lemma}
\begin{proof}
The ``only if" direction is given by Lemma~\ref{lem:sep-dense}, applied to the families $\Bb(n)$. We prove the ``if" implication.

Suppose that $(X,\tau)$ is not definably separable, and let $\Aa=\{A_b : b\in B_{\Aa}\}$ be an infinite definable family of pairwise disjoint open sets. Let us fix the minimum $n$ such that the family $\{ X(n) \cap A_b : b \in B_{\Aa}\}$ is infinite (clearly $0 \leq n\leq \dim X$). Let $Z(n)$ be as described in the lemma. We reach a contradiction by showing that $\dim Z(n) > n$. 

Since $X(n) \subseteq cl (Z(n))$, observe that the family of pairwise disjoint open sets $\mathcal{Z}=\{ Z(n) \cap A_b : b \in B_{\Aa}\}$ is infinite. Now, for any $b \in B_{\Aa}$ note that, if $\dim(Z(n) \cap A_b) < n$, then either $n=0$ and $Z(n) \cap A_b=\emptyset$, or otherwise $n>0$ and $Z(n) \cap A_b \subseteq X(n-1)$. By choice of $n$, it follows that the subfamily of sets in $\mathcal{Z}$ of dimension less than $n$ is finite, since otherwise we would have $n>0$ and the family $\{X(n-1) \cap A_b : b\in B_{\Aa} \}$ would be infinite. Hence
the family $\{ Z(n) \cap A_b: b \in B_{\Aa},\, \dim(Z(n) \cap A_b) \geq n\}$ is an infinite definable family of pairwise disjoint subsets of $Z(n)$ of dimension at least $n$. By Lemma~\ref{lem:disjoint-union-dim}, we conclude that $\dim Z(n) > n$. 
\end{proof}

In the case where $(X,\tau)$ has small boundaries Lemma~\ref{lem:sep-dense} can be simplified to provide a more straightforward characterization of definable separability. 

\begin{proposition}
Let $\Mm$ be an o-minimal structure. Let $(X,\tau)$ be a definable topological space with small boundaries. That is, $\dim(cl(Y)\setminus Y) < \dim Y$ for every definable subset $Y\subseteq X$. We fix a definable basis for $\tau$ and, for every $0 \leq n\leq \dim X$, let $X(n)$ be the union of all basic open sets of dimension at most $n$. Then $(X,\tau)$ is definably separable if and only if $\dim X(n) \leq n$ for every $n$. 
\end{proposition}
\begin{proof}
Having small boundaries implies that, for every $n$, if $Z(n)$ is a definable dense subset of $X(n)$, then $\dim Z(n) = \dim cl(Z(n)) = \dim X(n)$. So the proposition follows from Lemma~\ref{lem:sep-X(n)}.
\end{proof}






We now prove our main theorem about definable separability. 

\begin{theorem}\label{thm:separability}
Let $\Mm$ be an o-minimal expansion of $(\mathbb{R},<)$. Let $(X,\tau)$ be a definable topological space. The following are equivalent.  
\begin{enumerate}[(1)]
\item \label{itm:separability_1} $(X,\tau)$ is definably separable. 
\item \label{itm:separability_2} $(X,\tau)$ is separable. 
\item \label{itm:separability_3} $(X,\tau)$ has the countable chain condition (\textbf{ccc}). 
\end{enumerate}
\end{theorem}
\begin{proof}
The implication~\ref{itm:separability_2}$\Rightarrow$\ref{itm:separability_3} is a simple known fact in general topology. The implication~\ref{itm:separability_3}$\Rightarrow$\ref{itm:separability_1} (or rather its contrapositive) follows easily from Lemma~\ref{lem:inf-family}. We prove~\ref{itm:separability_1}$\Rightarrow$\ref{itm:separability_2}. 

Suppose that $(X,\tau)$ is definably separable. We proceed by induction on $n=\dim X$. The base case $n=0$ is immediate. We assume that $n>0$. Let us fix a definable basis $\{U_b : b\in B\}$ for $\tau$. Let $C=\{ b \in B :  \dim U_b < n\}$ and consider the definable family of sets $\Aa=\{U_c : c \in C\}$. Let $Y=\cup \Aa$. By Lemma~\ref{lem:sep-dense} there exists a definable open set $Z\subseteq Y$ with $\dim Z < n$ and $Y \subseteq cl(Z)$. Since $Z$ is open, by Lemma~\ref{lem:sep-open-subspace} the subspace $(Z,\tau|_Z)$ is definably separable. By induction hypothesis we derive that $(Z,\tau|_Z)$ is separable.  
Let $D_1$ denote a countable dense subset of $(Z,\tau|_z)$. Since $Y\subseteq cl(Z)$ note that $Y \subseteq cl(D_1)$. 

Now recall the classical fact that $X$ with the Euclidean topology is separable. Let $D_2$ denote a countable subset of $X$ that is dense in the Euclidean topology. 
Set $D=D_1\cup D_2$. This is a countable set. We show that it is dense in $(X,\tau)$. 

Let $A$ be a non-empty open set in $(X,\tau)$. By passing to a subset if necessary we may assume that $A$ is definable. If $A\cap Y \neq \emptyset$ then $A\cap D_1\neq \emptyset$. If $A\cap Y = \emptyset$ then, by definition of $Y$, it must be that $\dim A=n=\dim X$. But then, by o-minimality (the Euclidean topology has small boundaries), the set $A$ has non-empty interior in the Euclidean topology on $X$, and consequently $A\cap D_2\neq \emptyset$.    
\end{proof}

We now turn our attention to proving that a definable topological space in an o-minimal expansion of $(\mathbb{R},<)$ is definably second-countable if and only if it is second-countable (Theorem~\ref{thm:2c}). In doing so we will use the next three lemmas.  

\begin{lemma}\label{lem:count_union}
Let $\Mm$ be an o-minimal expansion of $(\mathbb{R},<)$.
Let $X\subseteq \mathbb{R}^n$ be a non-empty definable set and $\Cc$ be countable family (not necessarily definable) of definable subsets of $X$, each of dimension less than $\dim X$. Then $\Cc$ is not a cover of $X$, i.e. $\cup \Cc \subsetneq X$.  
\end{lemma}
\begin{proof}
By o-minimal cell decomposition we can assume that both $X$ and every set in $\Cc$ is a cell. We may also assume that $\dim X \geq 1$, since otherwise $\Cc$ contains only the empty set and the result is trivial. We proceed by induction on $n$, where $X\subseteq \mathbb{R}^n$. In the case $n=1$ the sets in $\Cc$ are either the empty set or singletons, and $Y$ is an interval, so the result follows. We assume that $n>1$.

Consider the projection $\pi(X)$ of $X$ to the first $n-1$ coordinates.
Let $\Cc(0) \subseteq \Cc$ we the subfamily of cells $C \in \Cc$ that are graphs of partial functions on $\pi(X)$, and let $\Cc(1)= \Cc \setminus \Cc(0)$. We first consider the case where $X$ is the graph of a function on $\pi(X)$. In this case we have that $\Cc(1)=\emptyset$ and $\dim \pi(C) = \dim C < \dim X = \dim \pi(X)$ for all $C\in \Cc$. By induction hypothesis it holds that $\cup\{ \pi(C) : C\in \Cc\} \subsetneq \pi(X)$, and consequently $\cup \Cc \subsetneq X$. 

Now suppose that $X$ is not a cell given by the graph of a function on $\pi(X)$, in particular we have that $\dim \pi(X) = \dim(X)-1$. For every $C \in \Cc(1)$ it holds that $\dim \pi(C) = \dim(C) - 1 < \dim(X) -1 = \dim \pi(X)$. Hence by induction hypothesis we may fix a point $d\in \pi(X) \setminus \cup\{ \pi(C) : C\in \Cc(1) \}$. Note that, by assumptions on $X$, the fiber $X_d$ is an interval. Moreover, for every $C\in \Cc(0)$ the fiber $C_d$ is a singleton. Since $\Cc(0)$ is countable, there exists $e \in X_d \setminus \cup \{ C_d : C\in \Cc(0)\}$, and we conclude that $(d,e) \in X \setminus \cup \Cc$.   
\end{proof}

\begin{lemma}\label{lem:count_subbasis}
Let $(X,\tau)$ be a topological space and $Y\subseteq X$ be a subset with a countable $\tau$-basis. For any $\tau$-basis $\Bb$ for $Y$ there exists a countable subfamily $\Cc\subseteq \Bb$ that is also a $\tau$-basis for $Y$. 
\end{lemma}
\begin{proof}
Let $\Bb_\omega$ be a countable $\tau$-basis for $Y$. Let $\Hh \subseteq \Bb_\omega \times \Bb_\omega$ be the set of pairs $(U, U') \in \Bb_\omega \times \Bb_\omega$ such that there exists some $A\in \Bb$ satisfying that $U\subseteq A \subseteq U'$. We fix one such set $A=A(U',U'')\in \Bb$ for every pair $(U, U') \in \Hh$. Let $\Cc = \{ A(U, U') : (U, U') \in \Hh\} \subseteq \Bb$. The family $\Cc$ is clearly countable. We show that it is a $\tau$-basis for $Y$. 

Let $x\in Y$ and consider an open set $U'\ni x$. By passing to a smaller neighborhood if necessary we may assume that $U'\in \Bb_\omega$. By definition of $\Bb$ there exists some $A\in \Bb$ satisfying that $x\in A \subseteq U'$. By definition of $\Bb_\omega$ there exists some $U\in \Bb_\omega$ such that $x\in U \subseteq A$. Hence $U \subseteq A \subseteq U'$. In particular, we have that $(U, U') \in \Hh$. Observe that $x \in A(U, U') \subseteq U'$. 
\end{proof}

The following is the easy direction of Theorem~\ref{thm:2c} below.

\begin{lemma}\label{lem:d2c-implies-2c}
Let $\Mm$ be an o-minimal expansion of $(\mathbb{R},<)$. Let $(X,\tau)$ be a definable topological space. If $(X,\tau)$ is definably second-countable then it is second-countable. 
\end{lemma}
\begin{proof}
We prove that every \nice{} subset of $X$ admits a countable $\tau$-basis. Let $Y\subseteq X$ be \nice{}, witnessed by a $\tau$-basis $\Bb = \{ U_b : b \in B\}$. That is, for every $y\in Y$ and $b\in B$ with $y\in U_b$, it holds that 
\begin{equation}\label{eqn:d2c-implies-2c}
\dim\{ c \in B : x \in U_c \subseteq U_b\} = \dim B. 
\end{equation}
Let $D$ be a countable subset of $B$ that is dense in $B$ in the Euclidean topology. For every $y\in Y$ and $b \in B$ with $y \in U_b$, equation~\eqref{eqn:d2c-implies-2c} implies that the set $\{ c \in B : x \in U_c \subseteq U_b\}$ has non-empty interior in $B$ in the Euclidean topology and so, by density of $D$, there exists some $c \in D$ such that $y \in U_c \subseteq U_b$. We have shown that $\{U_c : c \in D\}$ is a $\tau$-basis for $Y$. 
\end{proof}

We may now present our main result on definable second-countability. 

\begin{theorem}\label{thm:2c}
Let $\Mm$ be an o-minimal expansion of $(\mathbb{R},<)$. Let $(X,\tau)$ definable topological space. Then $(X,\tau)$ is definably second-countable if and only if it is second-countable.  
\end{theorem}
\begin{proof}
The ``only if" implication is given by Lemma~\ref{lem:d2c-implies-2c}. We prove the ``if" direction. Specifically, we prove that any definable subset of $X$ that has a countable $\tau$-basis is \nice{}. 

Let us fix a definable basis $\{U_b : b\in B\}$ for $\tau$, and definable sets $X'\subseteq X$ and $B'\subseteq B$ such that $\{ U_b : b\in B'\}$ is a $\tau$-basis for $X'$. We assume that $X'$ has a countable $\tau$-basis and prove that it is \nice{} by induction on $\dim X' + \dim B'$. If $\dim X' = 0$ then the result is given by Proposition~\ref{prop:bf}\ref{itm:prop-bf-2}, and if $\dim B'=0$ then the result is trivial, so we assume that $\dim X'>0$ and $\dim B' >0$. In particular this covers the base case of the induction. To ease notation we assume that $X' = X$ and $B'=B$.    


Let $H\subseteq X\times B \times B$ denote the relation where $(x,b,c) \in H$ whenever $x \in U_c \subseteq U_b$. 
For any $x\in X$, $b\in B$ and set $C'\subseteq B$ let $H(x,b;C')=\{ c\in C' : ( x,b,c ) \in H\}$.
For each $x\in X$, let $B(x)=\{ b\in B : 0 \leq \dim H(x,b;B) < \dim B\}$. Note that the sets $B(x)$ are definable uniformly in $x\in X$. Let $Y=\{ x\in X : B(x)\neq \emptyset\}$. 
Observe that the set $X\setminus Y$ is \nice{}, witnessed by $\{U_b : b\in B\}$. We assume that $Y\neq \emptyset$ and describe a finite partition of $Y$ into \nice{} sets, hence proving, by Lemma~\ref{lem:nice}, that $X$ is \nice{}. By passing if necessary to a set in a finite partition of $Y$, we may assume that $\dim B(x) = \dim B(y)$ for every $x, y \in Y$. The idea of the proof by induction is to find a definable subset $Z\subseteq Y$ such that $\dim Z < \dim Y$ and moreover $Y \setminus Z$ admits a definable $\tau$-basis of the form $\{U_b : b\in C \subseteq B\}$ where $\dim C < \dim B$.

Onwards for any set $F'\subseteq X \times B$ and $c\in B$ let $H(F';c)=\{( x,b ) \in F' : ( x,b,c ) \in H\}$. Let $F=\cup_{y\in Y} \{y\} \times B(y)$. 
Let $C$ be the set of all $c\in B$ such that  
$
\dim H(F;c) = \dim F. 
$ 
By Fact~\ref{fact:fiber-max}, applied to the set $H \cap (F\times C )$, we derive that 
\[
\dim C = \max_{( y,b ) \in F} \dim H(y,b;C).
\]
On the other hand by definition of $F$ every $(y, b) \in F$ satisfies that $\dim H(y,b;B)< \dim B$. It follows that $\dim C < \dim B$. 

Now let $Z\subseteq Y$ denote the set of points $z\in Y$ such that $\{ U_b : b\in C\}$ is not a $\tau$-basis for $\{z\}$. Note that the family $\{ U_b : b \in B\setminus C\}$ is a $\tau$-basis for $Z$, and $\{U_b: b\in C\}$ is a $\tau$-basis for $Y\setminus Z$. By the latter, and since $\dim C < \dim B$, we derive from the induction hypothesis that $Y\setminus Z$ is \nice{}. 

We complete the proof by showing that $\dim Z < \dim Y$. We may then apply the induction hypothesis to derive that $Z$ is also \nice{}, completing the proof of the theorem. We show that $\dim Z < \dim Y$ by assuming that $\dim Z = \dim Y$ and deriving a contradiction using the fact that $X$ has a countable $\tau$-basis.

Let $G=\cup_{z\in Z} \{z\}\times B(z) \subseteq F$. Recall that, by assumption on $Y$, it holds that $\dim B(x) = \dim B(y)$ for every $x, y \in Y$. Since by assumption $\dim Z = \dim Y$, it follows from Fact~\ref{fact:fiber} that $\dim G= \dim F$. 
Applying the definition of $C$ we reach that every $b\in B\setminus C$ satisfies that 
\begin{equation}\label{eqn:thm-2c-G}
\dim H(G;b) \leq \dim H(F;b) < \dim F = \dim G.
\end{equation}

Finally, since $\{ U_b : b \in B\setminus C\}$ is a $\tau$-basis for $Z \subseteq X$, by Lemma~\ref{lem:count_subbasis} there exists a countable set $D\subseteq B\setminus C$ such that $\{U_b : b\in D\}$ is a $\tau$-basis for $Z$. Observe that, for any $z\in Z$ and $b\in B(z)$, by definition of $\tau$-basis there must exist some $c\in D$ such that $z\in U_c \subseteq U_b$, meaning that $(z, b, c) \in H$. Hence the sets $\{ H(G;b) : b\in D\}$ cover $G$. By Lemma~\ref{lem:count_union} and equation~\eqref{eqn:thm-2c-G} we reach a contradiction. 
\end{proof}

In the following lemma we extract from the proof of Theorem~\ref{thm:2c} that the property of not being definably second-countable can be expressed with a single first-order formula. We use this later in Proposition~\ref{prop:families} to show that definable second-countability is definable in families. 

\begin{lemma}\label{lem:no-sc}
A definable topological space $(X,\tau)$ is definably second-countable if and only if there does not exists a non-empty definable set $Z\subseteq X$, two definable families of open sets $\{U_{b} : b \in B\}$ and $\{V_b : b\in B'\}$, and a definable set $G\subseteq Z\times B$, with the following three properties.
\begin{enumerate}
    \item For every $(z,b)\in G$ it holds that $z \in U_b$. 
    \item $\{V_b : b\in B'\}$ is a $\tau$-basis for $Z$. 
    \item \label{itm:not-2c-3} For every $b' \in B'$ it holds that 
    \[
    \dim \{(z,b) \in G : z \in V_{b'} \subseteq U_b\} < \dim G.
    \]
\end{enumerate}
\end{lemma}
\begin{proof} For the ``if" implication, suppose that $(X,\tau)$ does not admit a construction as described in the lemma. We may then show that $(X,\tau)$ is definably second-countable by following the proof of Theorem~\ref{thm:2c} up to equation~\eqref{eqn:thm-2c-G}, and then reaching a contradiction using the fact that, by assumption, $Z$, $G$, $\{U_{b} : b \in B\}$, and $\{U_{b} : b \in B\setminus C\}$ as described in said proof cannot exist.     

We now prove the ``only if" implication. Hence let $Z \subseteq X$, $\{U_{b} : b \in B\}$, $\{V_b : b\in B'\}$, and $G\subseteq Z \times B$ be as described in the lemma. Towards a contradiction suppose that $(X,\tau)$ is definably second-countable, witnessed by a definable basis $\{ A_c : c \in C\}$ for $\tau$. 


Let us define $E \subseteq Z \times B \times C \times C$ to be the relation given by $(z,b,c,c') \in E$ whenever
$
z \in A_c \subseteq A_{c'} \subseteq U_b. 
$
For any $c,c'\in C$, let $E(G;c,c')=\{ (z,b) \in G : (z,b,c,c')\in E\}$. Let $D = \{ (c,c') \in C \times C : \dim E(G;c,c') < \dim G\}$. 
Furthermore, for each $(z,b)\in G$ let $E(z,b;D)=\{ (c,c')\in D : (z,b,c,c')\in E\}$. We prove that, for every $(z,b)\in G$, it holds that $0\leq \dim E(z,b;D) = \dim D$. By Fact~\ref{fact:fiber-max} it then follows that there exists some $(c,c')\in D$ such that $\dim E(G;c,c') = \dim G$, which contradicts the definition of $D$. 

Hence let us fix a pair $(\mathbf{z},\mathbf{b})\in G$. Let $C' = \{ c' \in C : \mathbf{z} \in A_{c'} \subseteq U_{\mathbf{b}}\}$. Observe that, since $\{ A_c : c \in C\}$ witnesses that $(X,\tau)$ is definably second-countable (see Remark~\ref{rem:nice}), it holds that $\dim C' = \dim C$. 

\begin{claim}\label{claim:no-2c}
For each $c'\in C'$, it holds that $\dim \{c \in C : (c,c')\in E(\mathbf{z},\mathbf{b};D)\} = \dim C$. 
\end{claim}
\begin{claimproof}
Let us fix $c'\in C'$. Since $\{ V_{b'} : b'\in B'\}$ is a $\tau$-basis for $Z$, there exists some $b'\in B'$ such that $\mathbf{z} \in V_{b'} \subseteq A_{c'}$. Let $C''=\{ c \in C : \mathbf{z} \in A_c \subseteq V_{b'}\}$. Since $\{ A_c : c \in C\}$ witnesses that $(X,\tau)$ is definably second-countable (see Remark~\ref{rem:nice}), note that $\dim C'' =\dim C$. We show that $C''\times \{c'\} \subseteq E(\mathbf{z},\mathbf{b};D)$.

For any $c \in C''$, since $A_c \subseteq V_{b'} \subseteq A_{c'}$, observe that $E(G;c,c')\subseteq \{ (z,b) \in G : z \in V_{b'}\subseteq U_b\}$. Condition~\ref{itm:not-2c-3} in the lemma states that the latter set has dimension less than $\dim G$, and so we derive that $(c,c')\in D$. 
We have shown that
$
C'' \subseteq  \{c \in C : (c,c')\in E(\mathbf{z},\mathbf{b};D)\} \subseteq C.
$
Since $\dim C'' = \dim C$, the claim follows.  
\end{claimproof}

Using the fact that $\dim C' = \dim C$ and Claim~\ref{claim:no-2c}, together with Fact~\ref{fact:fiber}, we derive that 
\[
2\dim C =\dim \bigcup_{c' \in C'} \{c \in C : (c,c')\in E(\mathbf{z},\mathbf{b};D)\} \times \{c'\} \leq \dim E(\mathbf{z},\mathbf{b};D). 
\]
Since $E(\mathbf{z},\mathbf{b};D)\subseteq D \subseteq C \times C$, we conclude that $\dim E(\mathbf{z},\mathbf{b};D) = \dim D = 2\dim C \geq 0$. 
\end{proof}


A family of topological spaces $\{ (X_c,\tau_c) : c \in C\}$ is (uniformly) definable if there exists a partitioned formula $\varphi(x,y,z)$ and a definable set $B^{\text{opp}}\subseteq M^{|y|+|z|}$ with the following properties. The set $C$ is the projection of $B^{\text{opp}}$ to the last $|z|$ coordinates and, for every $c \in C$, the family of sets $\{ \varphi(M,b,c) : b \in B_c\}$, where $B_c=\{ b \in M^{|y|} : (b,c) \in B^{\text{opp}}\}$ and $\varphi(M,b,c)=\{ a \in M^{|x|} : \Mm \models \varphi(a,b,c)\}$, is a basis for the topology $\tau_c$. An example of a definable family of topological spaces would be given by any definable family of subsets of a given definable topological space, with the subspace topology.   

We show that the properties of definable separability and definable second-countability are definable in families. 


\begin{proposition}\label{prop:families}
Let $\Mm$ be an o-minimal structure. Let $\Cc=\{ (X_c,\tau_c) : c \in C\} $ be a definable family of topological spaces. There exists definable subsets $C_{\text{sep}}$ and $C_{\text{sc}}$ of $C$ such that, for every $c \in C$, the space $(X_c,\tau_c)$ is 
\begin{enumerate}[(1)]
\item \label{itm:fam-1} definably separable if and only if $c \in C_{\text{sep}}$, 
\item \label{itm:fam-2} definably second-countable if and only if $c \in C_{\text{sc}}$. 
\end{enumerate}
Moreover, $C_{\text{sep}}$ and $C_{\text{sc}}$ can be chosen definable over the same parameters as the family $\Cc$.
\end{proposition}
\begin{proof}

By Lemma~\ref{lem:sep-X(n)} and Definition~\ref{dfn:2nd-count} respectively, note that the properties of being definably separable and definably second-countable are maintained after passing to an elementary extension or substructure, and consequently we may assume that $\Mm$ is saturated. 

Let $C_{\text{sep}}$ denote the set of elements $c\in C$ such that the space $(X_c,\tau_c)$ is definably separable.
Definition~\ref{dfn:separable} and Lemma~\ref{lem:inf-family} show that a definable topological space is definably separable if and only if there does not exist an infinite definable family $\{A_t : t \in I\}$ of open subsets of $X$ such that $I\subseteq M$ is an interval and furthermore $A_t \cap A_{s}=\emptyset$ for every distinct pair $t, s \in I$. For any $c\in C$, let us say that a formula $\varphi(x,y,z)$, with $|y|=1$, \emph{witnesses that  $(X_c,\tau_c)$ is not definably separable}, if there exist some parameters $d\in M^{|z|}$, and some interval $I\subseteq M$, such that the family of sets $A_t=\{ a \in M^{|x|} : \Mm\models\varphi(a,t,d) \}$, for $t\in I$, has the properties described above. For every formula $\varphi(x,y,z)$ with $|y|=1$, observe that the set of all $c\in C$ such that $\varphi(x,y,z)$ witnesses that $(X_c,\tau_c)$ is not definably separable is definable over the same parameters as $\Cc$. Consequently, $C_{\text{sep}}$ is an intersection of sets definable over the same parameters as $\Cc$. Similarly, using Lemma~\ref{lem:sep-X(n)}, observe that $C \setminus C_{\text{sep}}$ is also an intersection of sets definable over the same parameters as $\Cc$. By a standard saturation argument we derive that $C_{\text{sep}}$ is definable over the same parameters as $\Cc$. 

The proof that $C_{\text{sc}}$, the set of all $c\in C$ such that $(X_c,\tau_c)$ is definably second-countable, is definable, is analogous, using Definition~\ref{dfn:2nd-count} and Lemma~\ref{lem:no-sc}.
\end{proof}

\begin{remark}\label{rem:Z}
Let $\Mm$ be an o-minimal structure and $(X,\tau)$ be a definable topological space. Let $\Aa=\{ A_b : b\in B\}$ be a definable family of open subsets of $X$ of dimension at most $n$. Observe that the proof of Lemma~\ref{lem:sep-dense} provides an explicit construction for a definable set $Z \subseteq \cup \Aa$, and shows that $Z$ is open in $X$ and dense in $\cup\Aa$. Furthermore, it applies the assumption that $(X,\tau)$ is definably separable to derive that $\dim Z \leq n$. Specifically, if $B\subseteq M$ then $Z$ is simply given by Equation~\eqref{eqn:sep-dense}, and in the general case the same idea is applied recursively. Note that the construction of $Z$ depends uniformly on the family $\Aa$. 

It follows that in Lemma~\ref{lem:sep-X(n)} one may fix the sets $Z(n)$, for every $n\leq \dim X$, to be the unique sets described in the proof of Lemma~\ref{lem:sep-dense} (for $\Aa = \Bb(n)$), and then $(X,\tau)$ is definably separable \textbf{if and only if} $\dim Z(n) \leq n$ for every $n < \dim X$. Furthermore, if $\{(X_c,\tau_c) : c \in C\}$ is a definable family of topological spaces, then, for each $n$, the corresponding sets $Z_c(n) \subseteq \Bb_c(n)$ are definable uniformly in $c \in C$. This can be used to yield a new proof of Proposition~\ref{prop:families}\ref{itm:fam-1}, which provides an explicit description of a formula defining the set $C_{\text{sep}}$.


A similar approach, using the proof of Theorem~\ref{thm:2c} and Lemma~\ref{lem:no-sc}, can be used to give an explicit description of a formula defining the set $C_{\text{sc}}$ in Proposition~\ref{prop:families}.
\end{remark}

\section{Definable metric spaces}\label{sec:metric}

In this section we explore definable separability and second-countability in the context of o-minimal definable metric spaces. Definable metric spaces were introduced and studied by Walsberg in~\cite{walsberg-thesis}. Although Walsberg works under the assumption that $\Mm$ is an o-minimal expansion of an ordered field, any notion that we borrow from~\cite{walsberg-thesis}, including the definition of definable metric space below, still makes sense in the ordered group setting. 

\begin{definition}\label{dfn:dms}

Let $\Mm=(M,0,+,<,\ldots)$ be an expansion of an ordered group. 
A \emph{definable metric space} is a tuple $(X,d)$, where $X$ is a definable set and \mbox{$d:X\times X\rightarrow M^{\geq 0}$} is a definable map that satisfies the metric axioms, namely identity of indiscernibles, symmetry and subadditivity. 
\end{definition}

Given a definable metric space $(X,d)$, $x \in X$, and $t \in M^{>0}$, we denote the open ball of center $x$ and radius $t$ by $B_d(x,t)=\{y\in X : d(x,y) < t\}$.
A definable metric space $(X,d)$ is a definable topological space with the topology generated by open balls. We denote this topology by $\tau_d$.

In~\cite{walsberg-thesis} Walsberg defined \emph{definable separability} among o-minimal definable metric spaces to be the property of not containing an infinite definable discrete subspace. His main result \cite[Theorem 9.0.1]{walsberg-thesis} implies that a metric space definable in an o-minimal expansion of the field of reals is definably separable if and only if it is separable. On the other hand, Walsberg's definition of definable separability is not suitable for general o-minimal (non-definably-metrizable) definable topological spaces, since these include for example the Moore plane \cite[Example A.12]{andujar_thesis}, which is definable in $(\mathbb{R},+,\cdot,<)$ and separable but nevertheless the subspace $\mathbb{R}\times \{0\}$ is discrete. Another example, this time one-dimensional, is given by \cite[Example A.9]{andujar_thesis}.

We describe the precise relationship between Walsberg's definition of definable separability (generalized to all definable topological spaces) and ours by means of the next definition, lemma, and proposition.

\begin{definition}\label{dfn:hereditary_separabiltiy}
A definable topological space $(X,\tau)$ is \emph{hereditarily definably separable} if every definable subspace of $(X,\tau)$ is definably separable (in the sense of Definition~\ref{dfn:separable}). 
\end{definition}

The next lemma states that, as long as $\Mm$ has definable choice, Walsberg's definition of definable separability is equivalent, in our terminology, to being hereditarily definably separable. 

\begin{lemma}\label{lemma:Walsberg_separability}
Suppose that $\Mm$ has definable choice. A definable topological space $(X,\tau)$ is hereditarily definably separable if and only if it does not contain an infinite definable discrete subspace. 
\end{lemma}
\begin{proof}
Any infinite discrete space is clearly not definably separable and so the ``only if" implication follows. 

For the ``if" implication, let $Y\subseteq X$ be a definable subset such that the subspace $(Y,\tau|_Y)$ is not definably separable. Let $\Aa=\{A_b : b \in B\}$ be an infinite definable family of pairwise disjoint open sets in $(Y,\tau|_Y)$. Using definable choice let $f:B\rightarrow \cup\Aa$ be a definable map such that $f(b)\in A_b$ for every $b\in B$, and satisfying moreover that $f(b)=f(c)$ whenever $A_b=A_c$. Then the image subspace $(f(B),\tau|_{f(B)})$ is an infinite definable discrete subspace.   
\end{proof}

Observe that, by Propositions~\ref{prop:bf}\ref{itm:prop-bf-1} and~\ref{prop:2c-implies-sep}, any definably second-countable definable topological space in an o-minimal structure is hereditarily definably separable.

Recall the fact from general topology that any separable metric space is hereditarily separable. We show that the same holds in the definable setting whenever the underlying structure has definable choice. 

\begin{proposition}\label{remark_def_sep}
Let $\Mm$ be an expansion of an ordered group with definable choice (e.g. $\Mm$ is an o-minimal expansion of an ordered group). Let $(X,d)$ be a definable metric space. The following are equivalent.
\begin{enumerate}[(1)]
\item \label{itm:def_sep_1} $(X,d)$ is definably separable.
\item \label{itm:def_sep_2} $(X,d)$ is hereditarily definably separable.
\item \label{itm:def_sep_3} $(X,d)$ does not contain an infinite definable discrete subspace (namely it is definably separable in the sense of Walsberg~\cite{walsberg-thesis}).
\end{enumerate}
\end{proposition}
\begin{proof}
The equivalence \ref{itm:def_sep_2}$\Leftrightarrow$\ref{itm:def_sep_3} is given by Lemma~\ref{lemma:Walsberg_separability}. The implication \ref{itm:def_sep_2}$\Rightarrow$\ref{itm:def_sep_1} is trivial. We prove \ref{itm:def_sep_1}$\Rightarrow$\ref{itm:def_sep_3} by contraposition.

If the ordered group structure on $\Mm$ is discrete then every definable metric space is discrete, and so definably separable if and only if finite. So we may assume that the ordered group is dense, i.e. $M^{>0}$ does not have a minimum. In particular, since for every $0<s<t$ it holds that either $2s\leq t$ or $2(t-s)\leq t$, observe that, for every $t>0$, there exists $s>0$ with $2s<t$. 

Let $(X,d)$ be a definable metric space and $Y$ be an infinite definable discrete subspace. By definable choice one may choose definably, for each $x\in Y$, some $\varepsilon_x>0$ such that $2\varepsilon_x < d(x,y)$ for every $y \in Y \setminus \{x\}$. We prove that the infinite definable family of open balls $\{B_d(x,\varepsilon_x) : x\in Y\}$ is pairwise disjoint, and so $(X,d)$ is not definably separable. 

Towards a contradiction suppose that there exists $x,y\in Y$ and some $z\in B_d(x,\varepsilon_x)\cap B_d(y,\varepsilon_y)$. Then by the triangle inequality 
$
d(x,y)\leq d(x,z)+d(z,y)\leq \varepsilon_x + \varepsilon_y \leq 2 \max\{\varepsilon_x, \varepsilon_y\}. 
$
Without loss of generality suppose that $\varepsilon_x=\max\{\varepsilon_x, \varepsilon_y\}$. Then this contradicts the fact that $2\varepsilon_x < d(x,y)$. 
\end{proof}

Since definable metric spaces are, by definition, defined only in expansions of ordered groups, Proposition~\ref{remark_def_sep} addresses all definable metric spaces in o-minimal structures. 

Onwards we restrict our scope exclusively to the o-minimal setting. Recall that in general topology a metric space is separable if and only if it is second-countable (if and only if it is hereditarily separable). We prove that this equivalence also holds in the o-minimal definable context. Joining this equivalence with the one in Proposition~\ref{remark_def_sep} we reach the main result of this section.

\begin{theorem}\label{thm:sep-metric-spaces}
Let $\Mm$ be an o-minimal expansion of an ordered group. Let $(X,d)$ be a definable metric space. The following are equivalent.
\begin{enumerate}[(1)]
\item \label{itm:thm-msp_1} $(X,d)$ is definably separable.

\item \label{itm:thm-msp_3} $(X,d)$ is hereditarily definably separable. 

\item \label{itm:thm-msp_2} $(X,d)$ does not contain an infinite definable discrete subspace (namely it is definably separable in the sense of Walsberg~\cite{walsberg-thesis}). 

\item \label{itm:thm-msp_4} $(X,d)$ is definably second-countable. 
\end{enumerate}
\end{theorem}

The equivalence between~\ref{itm:thm-msp_1}, \ref{itm:thm-msp_3} and \ref{itm:thm-msp_2} in Theorem~\ref{thm:sep-metric-spaces} is provided by Proposition~\ref{remark_def_sep}. Furthermore, the implication $\ref{itm:thm-msp_4} \Rightarrow\ref{itm:thm-msp_1}$ is given by Lemma~\ref{prop:2c-implies-sep}. Hence it remains to show that definable separability implies definable second-countability. We prove this in Lemma~\ref{lem:sep-implies-2c} below. 

In proving Lemma~\ref{lem:sep-implies-2c} we will use the next simple ad hoc lemma. It follows from an easy application of the triangle inequality, and so we leave the proof to the reader. Note that it does not necessitate that the metric map $d$ be definable.

\begin{lemma}\label{lem:easy-basis}
Let $\Mm$ be an expansion of an ordered divisible group and $(X,d)$ be a definable metric space. Let $X', Y\subseteq X$ be subsets with $X'\subseteq cl(Y)$. Then $\{ B_d(y,t) : y\in Y,\, t>0\}$ is a $\tau_d$-basis for $X'$. Specifically, for every $x\in X'$ and $t>0$, we have that $x\in B_d(y,s) \subseteq B_d(x,t)$ for every $y\in B_d(x,t/3) \cap Y$ and $t/3 < s < 2t/3$. 
\end{lemma}

In the proof of Lemma~\ref{lem:sep-implies-2c} below recall that, in every o-minimal expansion of an ordered group, the group is necessarily divisible. 

\begin{lemma}\label{lem:sep-implies-2c}
Let $\Mm$ be an o-minimal expansion of an ordered group and $(X,d)$ be a definable metric space. If $(X,d)$ is definably separable then it is definably second-countable. 
\end{lemma}
\begin{proof}
Let $(X,d)$ be a definably separable definable metric space. Let $\tau = \tau_d$. Our proof proceeds by showing that every definable subset $X'\subseteq X$ is \nice{} by induction on the dimension of a definable set $Y\subseteq X$ such that $X' \subseteq cl(Y)$. If $\dim Y = 0$ then, by Hausdorffness of definable metric spaces, $Y$ is also closed and so $X' \subseteq Y$ is finite and the result is given by Proposition~\ref{prop:bf}\ref{itm:prop-bf-2}. This covers the base case. Onwards suppose that $n=\dim Y>0$. To ease notation we consider $X'=X$. 

Let $Y'$ denote the union of all (basic) open definable sets of dimension at most $n-1$ in the subspace topology $\tau|_Y$.
Let $X_0= X \cap cl(Y\setminus Y')$ and $X_1=X\setminus X_0$. 
Since $X \subseteq cl(Y)$ we have that $X_1\subseteq cl (Y')$. Observe that, by Lemma~\ref{lem:sep-dense}, there exists a definable set $Z\subseteq Y'$ with $\dim Z \leq n-1$ and such that $Y' \subseteq cl(Z)$. In particular $X_1 \subseteq cl(Z)$. By the induction hypothesis we reach that $X_1$ is \nice{}. We end the proof by showing that $X_0$ is \nice{}, which completes the proof by application of Lemma~\ref{lem:nice}.

Let $\Bb_0=\{ B_d(y,t) : y \in Y,\, t>0\}$. By Lemma~\ref{lem:easy-basis} note that $\Bb_0$ is a $\tau$-basis for $X$ (in particular for $X_0$). We prove that $\Bb_0$ witnesses that $X_0$ is \nice{} by showing that, for every $x\in X_0$ and $t>0$, it holds that 
\[
\dim \{ ( y, s ) \in Y\times M^{>0} : x\in B_d(y,s)\subseteq B_d(x,t)\} = \dim (Y \times M^{>0}). 
\]
Hence let us fix $x\in X_0$ and $t>0$. By Lemma~\ref{lem:easy-basis} we have that 
\[
\{ ( y, s) : y\in Y\cap B_d(x,t/3),\, t/3<s<2t/3\} \subseteq \{ ( y, s ) \in Y\times M^{>0} : x\in B_d(y,s)\subseteq B_d(x,t)\}.
\]
Now notice that, by definition of the set $Y'$, any definable open set $A$ with $A\cap Y\setminus Y' \neq \emptyset$ satisfies that $\dim (A \cap Y) = n = \dim Y$. In particular, since $X_0= X \cap cl(Y\setminus Y')$ and $x\in X_0$, we have that  $\dim(Y\cap B_d(x,t/3))= \dim Y$. Applying Fact~\ref{fact:fiber}, we conclude that $\dim \{ ( y, s) : y\in Y\cap B_d(x,t/3),\, t/3<s<2t/3\} = \dim(Y) + 1 = \dim (Y \times M^{>0})$, as desired. 
\end{proof}

The above proof can be streamlined if one assumes that $(X,\tau)$ has small boundaries. This property was proved for definable metric spaces in o-minimal expansions of ordered fields in~\cite[Lemma 9.2.11]{walsberg-thesis}. Although it is likely that it holds for metric spaces definable in o-minimal expansions of groups too, the author is not aware of any proof of this. 

The Sorgenfrey line (Example~\ref{ex:sorg} defined in $(\mathbb{R},<)$) is a classical example of a (non-metrizable) topological space that is hereditary separable but not second-countable. The Moore plane and the Sorgenfrey plane are examples of (non-metrizable) topological spaces that are separable but not hereditarily separable, and in particular not second-countable. An analogous landscape is present in the o-minimal definable setting. That is, the three aforementioned examples are definable in the ordered field of reals, and furthermore have definable analogues in any o-minimal expansion of an ordered field (see e.g. Example A.12 in~\cite{andujar_thesis}) which display definable topological properties analogous to the ones described for the classical spaces.

\section{Towards a Definable Urysohn Metrization Theorem} \label{sec:conj}

We say that a definable topological space is \emph{affine} if it is definably homeomorphic to a set with the o-minimal Euclidean topology. The characterization of o-minimal affine definable topological spaces has been the subject of ongoing research in~\cite[Chapter 10]{dries98} (manifold and quotient spaces), \cite{walsberg15} (metric spaces), \cite{johnson14} (quotient spaces) and \cite{pet_rosel_18} and \cite{one-dim} (one-dimensional spaces). 

We end this paper by conjecturing that, in o-minimal expansions of ordered fields, definable second-countability characterizes affine definable topological spaces. This conjecture is informed by the contents of Chapter 7 and Example A.16 in~\cite{andujar_thesis}, the latter which describes a Hausdorff and regular definable topological space in $(\mathbb{R},+,\cdot,<)$ that can be partitioned into two definable subsets where the subspace topology is Euclidean, but nevertheless the space is not definably second-countable (in particular not affine). We conjecture that this condition is the only obstacle to achieve affineness. 

\begin{conjecture}[Definable Urysohn Metrization Conjecture] \label{conj:UMT}
Let $(X,\tau)$ be definable topological space in an o-minimal expansion of an ordered field. Then $(X,\tau)$ is definably homeomorphic to a set with the o-minimal Euclidean topology if and only if it is Hausdorff, regular, and definably second-countable. 
\end{conjecture}

The classical Urysohn Metrization Theorem states that every Hausdorff regular second-countable topological space is metrizable. 
The main result (Theorem 9.0.1) in~\cite{walsberg-thesis} states that a definable metric space in an o-minimal expansion of an ordered field is affine if and only if it is definably separable. By Proposition~\ref{prop:2c-implies-sep}, it follows that Conjecture~\ref{conj:UMT} has a positive answer among definable metric spaces. Furthermore, in order to prove the conjecture it would suffice to show that any Hausdorff, regular, definably second-countable space $(X,\tau)$ is definably metrizable, hence its name. 

Theorem 9.1 in \cite{one-dim} states that a one-dimensional Hausdorff definable topological space in an o-minimal expansion of an ordered field is affine if and only if it does not contain a subspace definably homeomorphic to an interval with the discrete or Sorgenfrey line topologies. Since these topologies are not definably second-countable (see Example~\ref{ex:sorg} addressing the latter), it follows (applying Proposition~\ref{prop:bf}\ref{itm:prop-bf-1}) that Conjecture~\ref{conj:UMT} also has a positive answer among one-dimensional spaces. 

\bibliographystyle{alpha}
\bibliography{O-min-sep-2c}

\providecommand{\noopsort}[1]{}\newcommand{\SortNoop}[1]{}
\begin{thebibliography}{{\SortNoop{Dries}}vdD98}

\bibitem[AF11]{asch-fisher-11}
Matthias Aschenbrenner and Andreas Fischer.
\newblock Definable versions of theorems by {K}irszbraun and {H}elly.
\newblock {\em Proc. Lond. Math. Soc. (3)}, 102(3):468--502, 2011.

\bibitem[AG21]{andujar_thesis}
Pablo And\'ujar~Guerrero.
\newblock {\em Definable Topological Spaces in O-minimal Structures}.
\newblock PhD thesis, Purdue University, 2021.

\bibitem[AGT23]{one-dim}
Pablo And\'ujar~Guerrero and Margaret E.~M. Thomas.
\newblock One-dimensional definable topological spaces in o-minimal structures, 2023.
\newblock arXiv:2310.04510.

\bibitem[AGTW21]{atw}
Pablo And\'{u}jar~Guerrero, Margaret E.~M. Thomas, and Erik Walsberg.
\newblock Directed sets and topological spaces definable in o-minimal structures.
\newblock {\em J. Lond. Math. Soc. (2)}, 104(3):989--1010, 2021.

\bibitem[BBT23]{gaga23}
Benjamin Bakker, Yohan Brunebarbe, and Jacob Tsimerman.
\newblock o-minimal {GAGA} and a conjecture of {G}riffiths.
\newblock {\em Invent. Math.}, 232(1):163--228, 2023.

\bibitem[DG22]{dolich-good-22}
Alfred Dolich and John Goodrick.
\newblock Tame topology over definable uniform structures.
\newblock {\em Notre Dame J. Form. Log.}, 63(1):51--79, 2022.

\bibitem[{\SortNoop{Dries}}vdD98]{dries98}
Lou {\SortNoop{Dries}}~van~den Dries.
\newblock {\em Tame Topology and O-minimal Structures}, volume 248 of {\em London Mathematical Society Lecture Note Series}.
\newblock Cambridge University Press, Cambridge, 1998.

\bibitem[FZ80]{FZ}
J\"org Flum and Martin Ziegler.
\newblock {\em Topological model theory}, volume 769 of {\em Lecture Notes in Mathematics}.
\newblock Springer, Berlin, 1980.

\bibitem[Joh18]{johnson14}
Will Johnson.
\newblock Interpretable sets in dense o-minimal structures.
\newblock {\em J. Symb. Log.}, 83(4):1477--1500, 2018.

\bibitem[Pil88]{pillay88}
Anand Pillay.
\newblock On groups and fields definable in {$o$}-minimal structures.
\newblock {\em J. Pure Appl. Algebra}, 53(3):239--255, 1988.

\bibitem[PR20]{pet_rosel_18}
Ya'acov Peterzil and Ayala Rosel.
\newblock Definable one-dimensional topologies in o-minimal structures.
\newblock {\em Arch. Math. Logic}, 59(1-2):103--125, 2020.

\bibitem[PS86]{pillay86}
Anand Pillay and Charles Steinhorn.
\newblock Definable sets in ordered structures. {I}.
\newblock {\em Trans. Amer. Math. Soc.}, 295(2):565--592, 1986.

\bibitem[PS01]{pet-star-01}
Ya'acov Peterzil and Sergei Starchenko.
\newblock Expansions of algebraically closed fields in o-minimal structures.
\newblock {\em Selecta Math. (N.S.)}, 7(3):409--445, 2001.

\bibitem[PW06]{pil-wilkie-06}
J.~Pila and A.~J. Wilkie.
\newblock The rational points of a definable set.
\newblock {\em Duke Math. J.}, 133(3):591--616, 2006.

\bibitem[SW19]{sim-wal-19}
Pierre Simon and Erik Walsberg.
\newblock Tame topology over dp-minimal structures.
\newblock {\em Notre Dame J. Form. Log.}, 60(1):61--76, 2019.

\bibitem[Tho12]{thomas12}
Margaret E.~M. Thomas.
\newblock Convergence results for function spaces over o-minimal structures.
\newblock {\em J. Log. Anal.}, 4(1), 2012.

\bibitem[Wal15a]{walsberg-thesis}
Erik Walsberg.
\newblock {\em Metric {G}eometry in a {T}ame {S}etting}.
\newblock ProQuest LLC, Ann Arbor, MI, 2015.
\newblock Thesis (Ph.D.)--University of California, Los Angeles.

\bibitem[Wal15b]{walsberg15}
Erik Walsberg.
\newblock On the topology of metric spaces definable in o-minimal expansions of fields, 2015.
\newblock arXiv:1510.07291.

\end{thebibliography}

\end{document}